\documentclass{tac}
\usepackage{amsmath,times}
\usepackage{amsfonts}
\usepackage{amssymb}
\usepackage{mathrsfs}
\usepackage{graphicx}
\usepackage{eucal,}
\usepackage[all,2cell]{xy}
\usepackage[shortlabels]{enumitem}	

\usepackage[
colorlinks=true]{hyperref}			
\hypersetup{allcolors=[rgb]{0.1,0.1,0.4}}

\author{Ronald Brown, George Janelidze, and George Peschke}
\thanks{The second author is partially supported by South African NRF.}
\address{Ronald Brown: School of Computer Science, \CR Bangor University, Bangor, United Kingdom \\ George Janelidze: Department of Mathematics and Applied Mathematics, \CR University of Cape Town, Rondebosch 7700, South Africa \\ George Peschke: Department of Mathematical and Statistical Sciences, \CR University of Alberta, Edmonton, Canada}

\title{Van Kampen's theorem for locally sectionable maps}
\copyrightyear{2020}


\eaddress{ronnie.profbrown@btinternet.com\CR george.janelidze@uct.ac.za\CR george.peschke@ualberta.ca}

\begin{document}

\maketitle
\begin{abstract}
We generalize the Van Kampen theorem for unions of non-connected spaces, due to R. Brown and A. R. Salleh, to the context where families of subspaces of the base space $B$ are replaced with a `large' space $E$ equipped with a \textit{locally sectionable} continuous map $p:E\to B$.
\end{abstract}

\section{Introduction}
\label{sec-Introduction}

The most general purely topological version of the 1-dimensional theorem of Van Kampen seems to be the following, which involves the fundamental groupoid $\pi_1(B,S)$ on a set $S$ of base points, cf.  \cite{BLMS,B,Indag}.  Here we write also $\pi_1(B,S)$ for  $\pi_1(B, B  \cap S)$, so that if $U$ is a subspace  of $B$ then $\pi_1(U, S)$ is $\pi_1(U, U \cap S)$.

\begin{theorem}[\cite{BS}]
\label{thm:basic vKT}
Let $(B_\lambda)_{\lambda\in \Lambda}$ be a family of subspaces of $B$ such that the interiors of the sets $B_\lambda$ $(\lambda\in \Lambda)$ cover $B$, and let $S$ be a subset of $B$. Suppose $S$ meets each path-component of each one-fold, two-fold, and each three-fold intersection of distinct members of the family $(B_\lambda)_{\lambda\in \Lambda}$. Then there  is a coequalizer diagram in the category of groupoids:
	\begin{equation}\label{equn:1coequ}
	\xymatrix{\bigsqcup_{\lambda,\mu\in\Lambda}\pi_1(B_\lambda\cap B_\mu, S) \ar@<0.5ex>[r]^-{\alpha}\ar@<-0.5ex>[r]_-{\beta}&\bigsqcup_{\lambda\in\Lambda}\pi_1(B_\lambda, S)\ar[r]^-{\gamma}&\pi_1(B,S),}
	\end{equation}
	in which $\bigsqcup$ stands for the coproduct in the category of groupoids, and $\alpha$, $\beta$, and $\gamma$ are determined by the inclusion maps $B_\lambda\cap B_\mu\to B_\lambda$, $B_\lambda\cap B_\mu\to B_\mu$, and $B_\lambda\to B$, respectively.
\end{theorem}

The idea for using more than one base point arose in the writing of the first version of \cite{ B} in order to  give a Van Kampen Theorem  general enough to  compute the fundamental group of the circle $S^1$,  which after all is  {\bf the} basic example in algebraic topology; of course the ``canonical" method for the circle is to use the covering map $\mathbb R \to S^1$,  but Theorem \ref{thm:basic vKT} deals also with a myriad of other cases.\footnote{We refer the reader to \url{https://mathoverflow.net/questions/40945/} for a discussion of the use of more than one base point, a notion   commonly not referred to in algebraic topology texts. }

   The use of many base points  has been  supported by Grothendieck in his 1984 ``Esquisse d'un Programme",  \cite[Section 2]{GrEsq}:
\begin{quote}
Ceci est li\'e notamment au fait que les gens s’obstinent encore, en calculant avec des groupes
fondamentaux, a fixer un seul point base, pl\^{u}tot que d’en choisir astucieusement tout un paquet
qui soit invariant par les symetries de la situation, lesquelles sont donc perdues en route. Dans
certaines situations (comme des th\'{e}or\`emes de descente \`a la Van Kampen pour les groupes fondamentaux) il est bien plus \'{e}l\'{e}gant, voire indispensable pour y comprendre quelque chose, de
travailler avec des groupo\"{\i}des fondamentaux par rapport \`a un paquet de points base convenable,....\end{quote}

And indeed, restricting to a single base point omits a vast range of other spaces available in a local-to-global way, that might require even a large number of base points, and new methods for analyzing the colimits involved. It is unreasonable to suggest that ``the answer has to be a group'',  or some form of presentation of a group; cf for example \cite{GGC}.

The use of groupoids also allows for the applications  of Higgins'  ``universal morphisms", see   \cite{H} and the Appendix B5 {\it Groupoids bifibered over Sets} of \cite{BHS} ; the fundamental   group of the circle  may be obtained from the groupoid $\mathcal I =\pi_1([0,1], \{0,1\})$ by identifying $0$ and $1$, as there is such a ``universal morphism" from $\mathcal I$ to the group of integers. Applications of groupoid notions to orbit spaces are given in the (current)  version of \cite[Chapter 11]{B}.

There are also accounts of the Van Kampen theorem for the fundamental group related to universal covers and Galois theory, \cite{D, BJ}; \cite{BJ} uses Galois theory of \cite{J}. In this paper we return to the classical approach, but replace families of subspaces of the base space $B$ with a `large' space $E$ equipped with a \textit{locally sectionable} continuous map $p:E\to B$; this makes our Main Theorem \ref{thm:main},  more general than Theorem \ref{thm:basic vKT}. In Section \ref{sec:MainTheorem} we prove Theorem \ref{thm:main} directly, and then, after making various additional remarks in Section \ref{sec:AdditionalRemarks}, briefly discuss a possibility of deducing it from Theorem \ref{thm:basic vKT} in Section \ref{sec:(1.1)->(2.2)}; for instance such deduction is obviously possible when all points of $B$ are taken as base points.

\section{Locally Sectionable Maps and the Main Theorem}
\label{sec:MainTheorem}

A continuous map $p:E\to B$ of topological spaces is said to be \textit{locally sectionable} if for every $b\in B$ there exist an open subset $U$ of $B$ such that $b\in U$ and the map $p^{-1}(U)\to U$ induced by $p$ has a continuous section. We are interested in such maps for two reasons:
\begin{enumerate}[$\bullet$]
		\item Let $(B_\lambda)_{\lambda\in \Lambda}$ be a family of subspaces of $B$, let $E$ be their coproduct, and suppose that  the interiors of the sets $B_\lambda$ $(\lambda\in \Lambda)$ cover $B$. Then the canonical map $p:E\to B$ is locally sectionable, and such a family of subspaces is used in the Brown\textendash Salleh version of van Kampen Theorem \cite{BS}, recalled above as Theorem \ref{thm:basic vKT}; we are going to extend that theorem to arbitrary locally sectionable maps. For convenience, for any map $p: E \to B$ and subset $S$ of $B$ we write $\pi_1(E,S)$ and $\pi_1(E\times_B E,S)$ for $\pi_1(E,E \times_B S)$ and $\pi_1(E\times_BE,E\times_BE\times_BS)$, respectively.
\item As shown in \cite{JT}, every locally sectionable map is an effective descent morphism in the category of topological spaces.
\end{enumerate}
Such maps also occur as surjective submersions in the context of \cite{NS}.

\begin{remark}
	$(\mathrm{a})$ Let us remove the word ``distinct'' from our requirement on $S$ in Theorem \ref{thm:basic vKT}. This will:
	\begin{enumerate}[$\bullet$]
		\item also remove ``one-fold, two-fold and'', since it will make one-fold and two-fold intersections special cases of three-fold intersections;
		\item make so modified Theorem \ref{thm:basic vKT} formally weaker by forcing each $U_\lambda$ to become path-connected.
	\end{enumerate}
	However, this formally weaker version of Theorem \ref{thm:main}  is actually easily equivalent to it.\\
	$(\mathrm{b})$ Consider the above-mentioned coproduct $E$ of the family $(B_\lambda)_{\lambda\in \Lambda}$ and the associated function $p: E  \to  B $. Since the fundamental groupoid functor $\pi_1$ preserves coproducts, we can write $\bigsqcup_{\lambda\in\Lambda}\pi_1(B_\lambda,S)=\pi_1(E,S)$ and
	\begin{equation*}
	\bigsqcup_{\lambda,\mu\in\Lambda}\pi_1(B_\lambda\cap B_\mu,  S)=\pi_1(E\times_BE,S)
	\end{equation*}
	which immediately suggests the formulation of the following  Main Theorem \ref{thm:main}. The proof, as in \cite{BS}, is by verification of the universal property; for this purpose,  we do not even need to know  that coequalisers exist in the category of groupoids, nor how to construct them. However the details of the construction may be found in \cite{H,BHS}.
\end{remark}

\begin{theorem}
\label{thm:main}%
Let $p:E\to B$ be a locally sectionable continuous map of topological spaces, and $S$ a subset of $B$ such that the inverse image $E\times_BE\times_BE\times_BS$ of $S$ under the canonical map $E\times_BE\times_BE\to B$ meets every path-component of $E\times_BE\times_BE$. Then there  is a coequalizer diagram in the category of groupoids
\begin{equation*}
\xymatrix{\pi_1(E\times_BE,S)\ar@<0.5ex>[r]^-{\alpha}\ar@<-0.5ex>[r]_-{\beta}&\pi_1(E,S)\ar[r]^-{\gamma}&\pi_1(B,S)},
\end{equation*}
in which the functor $\gamma$ is induced by $p:E \to B$ and the functors $\alpha$, $\beta$  are induced  respectively by the first and the second projections $p_1, p_2:E\times_BE\to E$.
\end{theorem}

Note that our assumption on path-components of $E\times_BE\times_BE$ implies similar assumptions on path-connected components of $E$ and of $E\times_BE$.\vspace{2.5mm}

The rest of this section, where we will use notation and assumptions of Theorem \ref{thm:main}, will be devoted to its proof.
\vspace{2.5mm}

Let $\mathcal{U}$ be the set of all open subsets $U$ of $B$ such that the map $p^{-1}(U)\to U$ induced by $p$ has a continuous section, and let us fix a $\mathcal{U}$-indexed family
\begin{equation*}
(s_U:U\to p^{-1}(U))_{U\in \mathcal{U}}
\end{equation*}	
of such sections. After that we introduce

\begin{definition}
\label{def:wpath}
A weighted path (in $(B,S)$) is a system $(f,n,\underline{t},\underline{U},\underline{g})$, in which:
\begin{enumerate}[$\bullet$]
	\item $f:[0,1]\to B$ is a path in $B$ with $f(0)$ and $f(1)$ in $S$;
	\item $n$ is a (non-zero) natural number;
	\item $\underline{t}=(t_0,\ldots,t_n)$ is a sequence of $n+1$ real numbers with $0=t_0<t_1<\ldots<t_n=1$;
	\item $\underline{U}=(U_1,\ldots,U_n)$ is a sequence of $n$ open sets from $\mathcal{U}$ with $f([t_{i-1},t_i])\subseteq U_i$, or, equivalently $[t_{i-1},t_i]\subseteq f^{-1}(U_i)$, for each $i=1,\ldots,n$; the map $[t_{i-1},t_i]\to U_i$ induced by $f$ will be denoted by $f_{\underline{U},i}$, or simply by $f_i$, if there is no danger of confusion;
	\item $\underline{g}=(g_1,\ldots,g_{n-1})$ is a sequence of $n-1$ paths $[0,1]\to E\times_BE$ with
	\begin{equation*}
	g_i(0)=(s_{U_i}f(t_i),s_{U_{i+1}}f(t_i))
	\end{equation*}
	and $g_i(1)$ in the inverse image $(E\times_BE)\times_BS$ of $S$ under the canonical map $E\times_BE\to B$, for each $i=1,\ldots,n-1$.
\end{enumerate}
We will also say that $(n,\underline{t},\underline{U},\underline{g})$ is a weight of $f$.
\end{definition}

From the definition of $\mathcal{U}$, the fact that the inverse image $E\times_BE\times_BS$ of $S$ under the canonical map $E\times_BE\to B$ meets every path-component of $E\times_BE$, and the fact that $[0,1]$ is a compact space, we obtain:

\begin{lemma}
\label{lem:weight}
	Every path $f:[0,1]\to B$ in $B$, with $f(0)$ and $f(1)$ in $S$, has a weight. Moreover, for every non-zero natural number $n$ and every sequence of $n+1$ real numbers with $0=t_0<t_1<\ldots<t_n=1$ and each $f([t_{i-1},t_i])$ $(i=1\ldots n)$ in some $U\in\mathcal{U}$, a weight $(n,\underline{t},\underline{U},\underline{g})$ for $f$ can be chosen with the same $n$ and $\underline{t}=(t_0,\ldots,t_n)$, and even with each $U_i$ chosen in advance.
\end{lemma}
	
Our next definition and the following lemmas show why weights are useful. Here and below, the morphism (in a suitable fundamental groupoid) determined by a path $f$ will be denoted by $[f]$ (not to be confused with, say, $[0,1]$ or $[t_{i-1},t_i]$). Moreover, we will do same when the domain of $f$ is not  just $[0,1]$ but any $[t,t']$ with $t<t'$.

\begin{definition}
\label{def:assocseq}%
For a weighted path $(f,n,\underline{t},\underline{U},\underline{g})$, the associated sequence of morphisms in $\pi_1(E,S)$ is the sequence $(h_n,\ldots,h_1)$ in which:
	\begin{enumerate}[$\bullet$]
		\item $h_1=[p_1g_1][s_{U_1}f_1]$;
		\item $h_i=[p_1g_i][s_{U_i}f_i][p_2g_{i-1}]^{-1}$, if $2\leqslant i\leqslant n-1$;
		\item $h_n=[s_{U_n}f_n][p_2g_{n-1}]^{-1}$.
	\end{enumerate}
	Here we use composition of maps inside square brackets and composition of morphisms in $\pi_1(E)$ outside them (in the order used in category theory).
\end{definition}

The following diagrams in the groupoid $\pi_1(E)$ (which are not diagrams in $\pi_1(E,S)$) help to understand this definition:

\begin{equation*}
\xymatrix{s_{U_1}f(0)=s_{U_1}f(t_0)\ar[r]^-{[s_{U_1}f_1]}&s_{U_1}f(t_1)\ar[r]^{[p_1g_1]}&p_1g_1(1),}
\end{equation*}
\begin{equation*}
\xymatrix{p_2g_{i-1}(1)&s_{U_i}f(t_{i-1})\ar[l]_-{[p_2g_{i-1}]}\ar[r]^-{[s_{U_i}f_i]}&s_{U_i}f(t_i)\ar[r]^-{[p_1g_i]}&p_1g_i(1),}
\end{equation*}
\begin{equation*}
\xymatrix{p_2g_{n-1}(1)&s_{U_n}f(t_{n-1})\ar[l]_-{[p_2g_{n-1}]}\ar[r]^-{[s_{U_n}f_n]}&s_{U_n}f(t_n)=s_nf(1).}
\end{equation*}

Let $G$ be an arbitrary groupoid and $\delta:\pi_1(E,S)\to G$ an arbitrary functor with $\delta\alpha=\delta\beta$.

\begin{lemma}
\label{lem:comp}
	Let $(h_n,\ldots,h_1)$ be as in Definition \ref{def:assocseq}. Then:
	\begin{enumerate}[$\bullet$]
		\item [(a)] The sequence $(\delta(h_n),\ldots,\delta(h_1))$ of morphisms in $G$ is composable; in particular this applies to $\delta=\gamma$, making $(\gamma(h_n),\ldots,\gamma(h_1))$ a composable sequence of morphisms in $\pi_1(B,S)$.
		\item [(b)] $\gamma(h_n)\ldots\gamma(h_1)=[f]$.
	\end{enumerate}
\end{lemma}
\begin{proof}
(a): We need to show that the domain of $\delta(h_{i+1})$ is the same as the codomain of $\delta(h_i)$, for every $i=1,\ldots,n-1$. Writing $\mathsf{dom}$ and $\mathsf{cod}$ for domain and codomain, respectively, we have:\vspace{1mm}
\begin{equation*}
\begin{array}{rcl}
\mathsf{dom}(\delta(h_{i+1})) & = & \delta(\mathsf{dom}(h_{i+1})) \\
	& = & \delta(\mathsf{dom}([p_1g_{i+1}][s_{U_{i+1}}f_{i+1}][p_2g_i]^{-1})) \\
	& \overset{\text{(i})}{=} & \delta(\mathsf{cod}([p_2g_i])) \\
	& \overset{\text{(ii)}}{=} & \delta(p_2g_i(1))=\delta(p_1g_i(1)) \\
	& = & \delta(\mathsf{cod}([p_1g_i]))=\delta(\mathsf{cod}([p_1g_i][s_{U_i}f_i][p_2g_{i-1}]^{-1}))=\delta(\mathsf{cod}(h_i))\ ,
\end{array}
\end{equation*}	
where equality (i) holds because all calculations inside $\pi_1(E,S)$ can equivalently be made inside $\pi_1(E)$; and equality (ii) holds since $(p_1g_i(1),p_2g_i(1))=g_i(1)$ belongs to $E\times_BE$ and $\delta\alpha=\delta\beta$.

\medskip\noindent
(b): Here we in fact need a calculation in $\pi_1(B,S)$, in which we can replace $\gamma:\pi_1(E,S)\to\pi_1(B,S)$ with $\pi_1(p):\pi_1(E)\to\pi_1(B)$. This allows us to present $\gamma(h_n)\ldots\gamma(h_1)$ as the composite of
	\begin{equation*}
	\xymatrix{f(0)=f(t_0)\ar[r]^-{[f_1]}&f(t_1)\ar[r]^-{[pp_1g_1]}&\ldots}
	\end{equation*}
	\begin{equation*}
	\xymatrix{\ldots\ar[r]^-{[pp_2g_{i-1}]^{-1}}& f(t_{i-1})\ar[r]^-{[f_i]}&f(t_i)\ar[r]^-{[pp_1g_i]}&pp_1g_i(1)=pp_2g_i(1)\ar[r]^-{[pp_2g_i]^{-1}}&f(t_i)\ar[r]^-{[f_{i+1}]}&f(t_{i+1})\ar[r]^-{[pp_1g_{i+1}]}&\ldots}
	\end{equation*}
	\begin{equation*}
	\xymatrix{\ldots\ar[r]^-{[pp_2g_{n-1}]^{-1}}&f(t_{n-1})\ar[r]^-{[f_n]}&f(t_n)=f(1),}
	\end{equation*}
	where we use the fact that, for each $i=1,\ldots,n$, the image of $[s_{U_i}f_i]:s_{U_i}f(t_{i-1})\to s_{U_i}f(t_i)$ under $\pi_1(p)$ is $[f_i]:f(t_{i-1})\to f(t_i)$ (since $s_U$ $(U\in\mathcal{U})$ are local sections of $p$). After that,  note also that $[pp_1g_i]=[pp_2g_i]$, and makes the composite above equal to $[f_n]\ldots[f_1]=[f]$.  	
\end{proof}

\begin{lemma}
\label{thm:Delta-Identity1}
Let $(f,n,\underline{t},\underline{U},\underline{g})$ and $(f,n,\underline{t},\underline{U},\underline{g'})$ be weighted paths and $(h_n,\ldots,h_1)$ and $(h'_n,\ldots,h'_1)$ their associated sequences of morphisms in $\pi_1(E,S)$. Then
\begin{equation*}
\delta(h_n)\ldots\delta(h_1)=\delta(h'_n)\ldots\delta(h'_1)\ .
\end{equation*}
\end{lemma}
\begin{proof}
	We can assume, without loss of generality, $g_i=g'_i$ for all $i=1,\ldots,n-1$ except $i=k$, for some $k$. Assuming that we only need to prove the equality $\delta(h_{k+1})\delta(h_k)=\delta(h'_{k+1})\delta(h'_k)$. There are three cases to consider, $k=1$, $2\leqslant k\leqslant n-2$, and $k=n-1$, but we will consider only second case (since two other cases can obviously be treated similarly). We have:
\begin{enumerate}[$\bullet$]
		\item [(a)] $\delta([p_1g_k][p_1g'_k]^{-1})=\delta\alpha([g_k][g'_k]^{-1}):\delta\alpha(g'_k(1))\to\delta\alpha(g_k(1))$ in $G$, where, although we compose morphisms in $\pi_1(E)$ and $\pi_1(E\times_BE)$, the results $[p_1g_k][p_1g'_k]^{-1}$ and $[g_k][g'_k]^{-1}$ belong to $\pi_1(E,S)$ and $\pi_1(E\times_BE,S)$, respectively;
		\item [(b)] similarly, $\delta([p_2g_k][p_2g'_k]^{-1})=\delta\beta([g_k][g'_k]^{-1}):\delta\beta(g'_k(1))\to\delta\beta(g_k(1))$ in $G$;
		\item [(c)] since $\delta\alpha=\delta\beta$, (a) and (b) imply $\delta([p_1g_k][p_1g'_k]^{-1})=\delta([p_2g_k][p_2g'_k]^{-1})$;
		\item [(d)] In $G$ we compute:
\begin{equation*}
\scalebox{0.885}{$\delta([p_2g_k][p_2g'_k]^{-1})\delta([p_1g'_k][p_1g_k]^{-1})=\delta([p_2g_k][p_2g'_k]^{-1})(\delta([p_1g_k][p_1g'_k]^{-1}))^{-1}=1_{\delta\alpha(g_k(1))}=1_{\delta\beta(g_k(1))}$}
\end{equation*}		
\end{enumerate}
Using (d) we calculate:
\begin{equation*}
\scalebox{0.9}{$\begin{array}{rcl}
\delta(h_{k+1})\delta(h_k) & = & \delta([p_1g_{k+1}][s_{U_{k+1}}f_{k+1}][p_2g_k]^{-1})\delta([p_1g_k][s_{U_k}f_k][p_2g_{k-1}]^{-1}) \\
	& = & \delta([p_1g_{k+1}][s_{U_{k+1}}f_{k+1}][p_2g_k]^{-1}[p_2g_k][p_2g'_k]^{-1})\delta([p_1g'_k][p_1g_k]^{-1}[p_1g_k][s_{U_k}f_k][p_2g_{k-1}]^{-1}) \\
	& = & \delta([p_1g_{k+1}][s_{U_{k+1}}f_{k+1}][p_2g'_k]^{-1})\delta([p_1g'_k][s_{U_k}f_k][p_2g_{k-1}]^{-1}) \\
	& = & \delta([p_1g'_{k+1}][s_{U_{k+1}}f_{k+1}][p_2g'_k]^{-1})\delta([p_1g'_k][s_{U_k}f_k][p_2g'_{k-1}]^{-1}) \\
	& = & \delta(h'_{k+1})\delta(h'_k).	
\end{array}$}
\end{equation*}
This completes the proof.
\end{proof}

The following lemma is stronger, but it uses Lemma \ref{thm:Delta-Identity1} in its proof.

\begin{lemma}
\label{thm:Delta-Identity2}%
Let $(f,n,\underline{t},\underline{U},\underline{g})$ and $(f,n',\underline{t'},\underline{U'},\underline{g'})$ be weighted paths and $(h_n,\ldots,h_1)$ and $(h'_{n'},\ldots,h'_1)$ their associated sequences of morphisms in $\pi_1(E,S)$. Then
\begin{equation*}
\delta(h_n)\ldots\delta(h_1)=\delta(h'_{n'})\ldots\delta(h'_1)\ .
\end{equation*}
\end{lemma}
\begin{proof}
We can assume, without loss of generality, that $\{t_0\ldots t_n\}\cap\{t'_0\ldots t'_{n'}\}=\{0,1\}$. Indeed, using the second assertion of Lemma 2.4, we can choose a weight $(n'',\underline{t''},\underline{U''},\underline{g''})$ for $f$ with $\{t_0\ldots t_n\}\cap\{t''_0\ldots t''_{n''}\}=\{0,1\}$ and $\{t'_0\ldots t'_{n'}\}\cap\{t''_0\ldots t''_{n''}\}=\{0,1\}$, and then $\delta(h_n)\ldots\delta(h_1)=\delta(h''_{n''})\ldots\delta(h''_1)$ and $\delta(h'_{n'})\ldots\delta(h'_1)=\delta(h''_{n''})\ldots\delta(h''_1)$ would imply $\delta(h_n)\ldots\delta(h_1)=\delta(h'_{n'})\ldots\delta(h'_1)$ (in obvious notation).
	
	After that let us introduce the following notation:
	\begin{enumerate}[$\bullet$]
		\item $\tilde{n}=n+n'-1$.
		\item $\{\tilde{t}_0,\ldots,\tilde{t}_{\tilde{n}}\}=\{t_0\ldots t_n\}\cup\{t'_0\ldots t'_{n'}\}$.
		\item $f_i$ is the restristion of $f$ on $[\tilde{t}_{i-1},\tilde{t}_i]$; accordingly, we will not abbreviate $f_{\underline{U},i}$ as $f_i$ in this proof. Note that: (a) $f_i=f_{\underline{U},i}$ when $\tilde{t}_i=t_j$ and $\tilde{t}_{i-1}=t_{j-1}$; (b) $f_i=f_{\underline{U'},i}$ when $\tilde{t}_i=t'_j$ and $\tilde{t}_{i-1}=t'_{j-1}$.
		\item
		\begin{equation*}
		\xymatrix{\{1,\ldots,n-1\}\ar@<0.5ex>[r]^\varphi&\{1,\ldots,n'-1\}\ar@<0.5ex>[l]^{\varphi'}}
		\end{equation*}
		are the maps defined by
		\begin{equation*}
		\varphi(i)=\mathrm{min}\{j\in\{1,\ldots,n'\}\,|\,t_i<t'_j\}\,\,\,\mathrm{and}\,\,\,\varphi'(i)=\mathrm{min}\{j\in\{1,\ldots,n\}\,|\,t'_i<t_j\}.	
		\end{equation*}
		\item $q_1$, $q_2$, and $q_3$ are the three projections $E\times_BE\times_BE\to E$.
	\end{enumerate}
    Let us also choose paths $r_i:[0,1]\to E\times_BE\times_BE$ $(i=1,\ldots,\tilde{n}-1)$ with
    \begin{equation*}
	r_i(0)=\left\{
	\begin{array}{l}
	(s_{U_j}f(t_j),s_{U'_{\varphi(j)}}f(t_j),s_{U_{j+1}}f(t_j)),\ \text{if}\ \tilde{t}_i=t_j,\\
	(s_{U'_j}f(t'_j),s_{U_{\varphi'(j)}}f(t'_j),s_{U'_{j+1}}f(t'_j)),\ \text{if}\ \tilde{t}_i=t'_j,
	\end{array}
	\right.		
	\end{equation*}
	and $r_i(1)$ in the inverse image $(E\times_BE\times_BE)\times_BS$ of $S$ under the canonical map
	\begin{equation*}
	E\times_BE\times_BE\to B,
	\end{equation*}
	for each $i=1,\ldots,\tilde{n}-1$; such choices are possible by our assumption on path-components of $E\times_BE\times_BE$ in Theorem \ref{thm:main}.
	
	Using this data we define morphisms $\tilde{h}_1,\ldots,\tilde{h}_{\tilde{n}}$ in $\pi_1(E\times_BE,S)$,  by
	\begin{equation*}
		\resizebox{.98\textwidth}{!}{$
	\tilde{h}_i=\left\{
	\begin{array}{l}
	[\langle q_1r_1,q_2r_1\rangle][\langle s_{U_1}f_1,s_{U'_1}f_1\rangle],\ \text{if}\ i=1\ \text{and}\ \tilde{t}_1=t_1,\\
	\lbrack\langle q_2r_1,q_1r_1\rangle][\langle s_{U_1}f_1,s_{U'_1}f_1\rangle],\ \text{if}\ i=1\ \text{and}\ \tilde{t}_1=t'_1,\\
	\lbrack\langle q_1r_i,q_2r_i\rangle][\langle s_{U_j}f_i,s_{U'_{\varphi(j)}}f_i\rangle][\langle q_3r_{i-1},q_2r_{i-1}\rangle]^{-1},\ \text{if}\ 1<i<\tilde{n},\, \tilde{t}_i=t_j,\, \text{and}\ \tilde{t}_{i-1}=t_{j-1},\\
	\lbrack\langle q_2r_i,q_1r_i\rangle][\langle s_{U_{\varphi'(j)}}f_i,s_{U'_j}f_i\rangle][\langle q_2r_{i-1},q_3r_{i-1}\rangle]^{-1},\ \text{if}\ 1<i<\tilde{n},\, \tilde{t}_i=t'_j,\, \text{and}\ \tilde{t}_{i-1}=t'_{j-1},\\
	\lbrack\langle q_1r_i,q_2r_i\rangle][\langle s_{U_j}f_i,s_{U'_{\varphi(j)}}f_i\rangle][\langle q_2r_{i-1},q_3r_{i-1}\rangle]^{-1},\ \text{if}\ 1<i<\tilde{n},\, \tilde{t}_i=t_j,\, \text{and}\ \tilde{t}_{i-1}=t'_{\varphi(j)-1},\\
	\lbrack\langle q_2r_i,q_1r_i\rangle][\langle s_{U_{\varphi'(j)}}f_i,s_{U'_j}f_i\rangle][\langle q_3r_{i-1},q_2r_{i-1}\rangle]^{-1},\ \text{if}\ 1<i<\tilde{n},\, \tilde{t}_i=t'_j,\, \text{and}\ \tilde{t}_{i-1}=t_{\varphi'(j)-1},\\
	\lbrack\langle s_{U_n}f_{\tilde{n}},s_{U'_{n'}}f_{\tilde{n}}\rangle][\langle q_3r_{\tilde{n}-1},q_2r_{\tilde{n}-1}\rangle],\ \text{if}\ i=\tilde{n},\, \text{and}\ \tilde{t}_{\tilde{n}-1}=t_{n-1},\\
	\lbrack\langle s_{U_n}f_{\tilde{n}},s_{U'_{n'}}f_{\tilde{n}}\rangle][\langle q_2r_{\tilde{n}-1},q_3r_{\tilde{n}-1}\rangle],\ \text{if}\ i=\tilde{n},\, \text{and}\ \tilde{t}_{\tilde{n}-1}=t'_{n-1}.\\
	\end{array}
	\right.	$}
	\end{equation*}
Here we omitted routine calculations to show that the composites involved are well defined in $\pi_1(E\times_BE)$, and similar routine calculations show that
\begin{equation*}
(\delta\alpha(\tilde{h}_{\tilde{n}}),\ldots,\delta\alpha(\tilde{h}_1))=(\delta\beta(\tilde{h}_{\tilde{n}}),\ldots,\delta\beta(\tilde{h}_1))
\end{equation*}
is a composable sequences of morphisms in $G$.

Our aim is to prove that
\begin{equation*}
\delta(h_n)\ldots\delta(h_1)=\delta\alpha(\tilde{h}_{\tilde{n}})\ldots\delta\alpha(\tilde{h}_1)=\delta\beta(\tilde{h}_{\tilde{n}})\ldots\delta\beta(\tilde{h}_1)=\delta(h'_n)\ldots\delta(h'_1),
\end{equation*}
but since the middle equality is trivial, while the first and third one are similar to each other, it suffices to prove the first equality.

Let $\chi:\{0,\ldots,n\}\to\{0,\ldots,\tilde{n}\}$ be the map defined by $\chi(j)=i\Leftrightarrow t_j=\tilde{t}_i$. We have
\begin{equation*}
\delta\alpha(\tilde{h}_{\tilde{n}})\ldots\delta\alpha(\tilde{h}_1)=\delta\alpha(\tilde{h}_{\chi(n)})\ldots\delta\alpha(\tilde{h}_{\chi(n-1)+1})\ldots\delta\alpha(\tilde{h}_{\chi_1})\ldots\delta\alpha(\tilde{h}_1),
\end{equation*}
and so to prove the desired first equality above it suffices to prove the equality
\begin{equation*}
\delta(h_j)=\delta\alpha(\tilde{h}_{\chi(j)})\ldots\delta\alpha(\tilde{h}_{\chi(j-1)+1}),\,\,\,\,\,\,\,\,(1)
\end{equation*}
for every $j=1,\ldots,n$. Let us do it assuming $1<j<n$ (since in the cases $j=1$ and $j=n$ it can be done similarly, but with a bit shorter calculation). For, putting $\chi(j)=i$, we observe:
\begin{enumerate}[(a)]
	\item If $\chi(j)=\chi(j-1)+1$, we have
	\begin{equation*}
	\begin{array}{rcl}
	\delta\alpha(\tilde{h}_{\chi(j)})\ldots\delta\alpha(\tilde{h}_{\chi(j-1)+1}) & = & \delta\alpha(\tilde{h}_{\chi(j)}) \\
		& = & \delta\alpha([\langle q_1r_i,q_2r_i\rangle][\langle s_{U_j}f_i,s_{U'_{\varphi(j)}}f_i\rangle][\langle q_3r_{i-1},q_2r_{i-1}\rangle]^{-1}) \\
		& = & \delta([q_1r_i][s_{U_j}f_i][q_3r_{i-1}]^{-1}),
	\end{array}
	\end{equation*}
	where the second equality follows from the third line in the definition of $\tilde{h}_i$.
	\item If $\chi(j)=\chi(j-1)+2$, which gives
	\begin{equation*}
	\tilde{t}_{\chi(j)-2}=\tilde{t}_{\chi(j-1)}=t_{j-1}\quad \text{and then}\quad \tilde{t}_{\chi(j)-1}=t'_{\varphi(j-1)}=t'_{\varphi(j)-1},
	\end{equation*}
	we have
	\begin{equation*}
	\resizebox{.945\textwidth}{!}{$
	\begin{array}{l}
	\delta\alpha(\tilde{h}_{\chi(j)})\ldots\delta\alpha(\tilde{h}_{\chi(j-1)+1})=\delta\alpha(\tilde{h}_{\chi(j)})\delta\alpha(\tilde{h}_{\chi(j)-1}) \\
	\quad =\delta\alpha([\langle q_1r_i,q_2r_i\rangle][\langle s_{U_j}f_i,s_{U'_{\varphi(j)}}f_i\rangle][\langle q_2r_{i-1},q_3r_{i-1}\rangle]^{-1})\delta\alpha(\tilde{h}_{\chi(j)-1}) \\
	\quad =\delta([q_1r_i][s_{U_j}f_i][q_2r_{i-1}]^{-1})\delta\alpha(\tilde{h}_{\chi(j)-1}) \\
	\quad =\delta([q_1r_i][s_{U_j}f_i][q_2r_{i-1}]^{-1})\delta\alpha([\langle q_2r_{i-1},q_1r_{i-1}\rangle][\langle s_{U_{\varphi'(\varphi(j)-1)}}f_{i-1},s_{U'_{\varphi(j)-1}}f_{i-1}\rangle][\langle q_3r_{i-2},q_2r_{i-2}\rangle]^{-1}) \\
	\quad =\delta([q_1r_i][s_{U_j}f_i][q_2r_{i-1}]^{-1})\delta([q_2r_{i-1}][ s_{U_{\varphi'(\varphi(j)-1)}}f_{i-1}][q_3r_{i-2}]^{-1}) \\
	\quad =\delta([q_1r_i][s_{U_j}f_i][q_2r_{i-1}]^{-1})\delta([q_2r_{i-1}][ s_{U_j}f_{i-1}][ q_3r_{i-2}]^{-1}) \\
	\quad =\delta([q_1r_i][s_{U_j}f_i][ s_{U_j}f_{i-1}][ q_3r_{i-2}]^{-1}),
	\end{array} $}
	\end{equation*}
	where:
	\begin{enumerate}[$-$]
		\item the second equality follows from the fifth line in the definition of $\tilde{h}_i$;
		\item the forth equality follows from the sixth line in the definition of $\tilde{h}_i$ with $i$ replaced with $i-1$;
		\item the sixth equality follows from the equation $\varphi'(\varphi(j)-1)=j$, while this equation is easy to check having in mind that $t'_{\varphi(j)-1}=\tilde{t}_{\chi(j)-1}$ and $t_j=\tilde{t}_{\chi(j)}$;
		\item the last equation follows from the fact that
		\begin{equation*}
		([q_1r_i][s_{U_j}f_i][q_2r_{i-1}]^{-1},[q_2r_{i-1}][ s_{U_j}f_{i-1}][ q_3r_{i-2}]^{-1})
		\end{equation*}
		is a composable pair of morphisms in $\pi_1(E,S)$.
	\end{enumerate}
	\item  Similarly, whenever $\chi(j)>\chi(j-1)+2$, one can show that $\delta\alpha(\tilde{h}_{\chi(j)})\ldots\delta\alpha(\tilde{h}_{\chi(j-1)+1})$ is given by
	\begin{equation*}
	\delta([q_1r_{\chi(j)}][s_{U_j}f_{\chi(j)}][ s_{U_j}f_{\chi(j)-1}]\ldots[ s_{U_j}f_{\chi(j)+2}][ s_{U_j}f_{\chi(j)+1}][ q_3r_{\chi(j)}]^{-1}).
	\end{equation*}
\end{enumerate}

Now, since (inside $\pi_1(E)$) we have $[s_{U_j}f_{\chi(j)}]\ldots[ s_{U_j}f_{\chi(j)+1}]=[s_{U_j}f_{\underline{U},j}]$, to prove $(1)$ it would suffice to prove the equality
	\begin{equation*}
	[p_1g_j][s_{U_j}f_{\underline{U},j}][p_2g_{j-1}]^{-1}=[q_1r_{\chi(j)}][s_{U_j}f_{\underline{U},j}][ q_3r_{\chi(j)}]^{-1}.\,\,\,\,\,\,\,\,(2)
	\end{equation*}
Moreover, thanks to Lemma \ref{thm:Delta-Identity1}, in  proving (2) we are allowed to change the component $\underline{g}$ of the weighted path $(f,n,\underline{t},\underline{U},\underline{g})$. But this makes the equality (2) trivial, since we can define $\underline{g}$ by $g_j=\langle q_1r_{\chi(j)},q_3r_{\chi(j)}\rangle$, which makes $p_1g_j=q_1r_{\chi(j)}$ and $p_2g_{j-1}=q_3r_{\chi(j)}$.
\end{proof}

\begin{definition}
	Let $P=(f,n,\underline{t},\underline{U},\underline{g})$ and $P'=(f',n',\underline{t'},\underline{U'},\underline{g'})$ be weighted paths with $f(0)=f'(0)$, $f(1)=f'(1)$, $n=n'$, $\underline{t}=\underline{t'}$, and $\underline{U}=\underline{U'}$. A homotopy $H:P\to P'$ is a continuous map $H:[0,1]\times[0,1]\to B$ with
	\begin{itemize}
		\item [(a)] $H(t,0)=f(t)$, $H(t,1)=f'(t)$, $H(0,t)=f(0)$, $H(0,t)=f(0)$ for every $t\in[0,1]$;
		\item [(b)] $H([t_{i-1},t_i]\times[0,1])\subseteq U_i$, for each $i=1,\ldots,n$.
	\end{itemize}
	We will say that $P$ and $P'$ are homotopic when such a homotopy $H:P\to P'$ does exist.	
\end{definition}

\begin{lemma}
\label{thm:Delta-Identity3}%
Let $(f,n,\underline{t},\underline{U},\underline{g})$ and $(f',n',\underline{t'},\underline{U'},\underline{g'})$ be homotopic weighted paths, and let $(h_n,\ldots,h_1)$ and $(h'_n,\ldots,h'_1)$ be their associated sequences of morphisms in $\pi_1(E,S)$. Then $\delta(h_n)\ldots\delta(h_1)=\delta(h'_{n'})\ldots\delta(h'_1)$.
\end{lemma}
\begin{proof}
	Thanks to Lemma \ref{thm:Delta-Identity1}, we can change $\underline{g}$ and $\underline{g'}$ as it will be suitable for our argument. We will not change $\underline{g'}$, but, given a homotopy $H:P\to P'$, assume that $[g_i]$ $(i=1,\ldots,n-1)$ is the composite
	\begin{equation*}
	\xymatrix{(s_{U_i}f(t_i),s_{U_{i+1}}f(t_i))\ar[r]^{[H_i]}&(s_{U_i}f'(t_i),s_{U_{i+1}}f'(t_i))\ar[r]^-{[g'_i]}&g'_i(1),}
	\end{equation*}
	where $H_i$ is a path in $E\times_BE$ defined by $H_i(t)=(s_{U_i}H(t_i,t),s_{U_{i+1}}H(t_i,t))$.
	
	For $2\leqslant i\leqslant n-1$, consider the diagram
	\begin{equation*}
	\xymatrix{p_2g'_{i-1}(1)&s_{U_i}f'(t_{i-1})\ar[l]_-{[p_2g'_{i-1}]}\ar[r]^-{[s_{U_i}f'_i]}&s_{U_i}f'(t_i)\ar[r]^-{[p_1g'_i]}&p_1g'_i(1)\\&s_{U_i}f(t_{i-1})\ar[r]_-{[s_{U_i}f_i]}\ar[u]^{[p_2H_{i-1}]}&s_{U_i}f(t_i)\ar[u]_{[p_1H_i]}}
	\end{equation*}
	in $\pi_1(E)$. Since its square part represents (up to ordinary homotopy) the boundary of the $s_{U_i}H$-image of the rectangle $[t_{i-1},t_i]\times[0,1]$, we obtain
	\begin{equation*}
	[s_{U_i}f'_i]=[p_1H_i][s_{U_i}f_i][p_2H_{i-1}]^{-1},
	\end{equation*}
	which gives $h_i=h'_i$, according to the second formula in Definition 2.5 and our choice of $\underline{g}$. Proving that $h_i=h'_i$ for $i=1$ and $i=n$ is similar (and a bit shorter). Hence $\delta(h_n)\ldots\delta(h_1)=\delta(h'_{n'})\ldots\delta(h'_1)$, as desired.
\end{proof}

\begin{lemma}
	Let $f$ and $f'$ be paths in $B$ with endpoints in $S$ and $[f]=[f']$ in $\pi_1(B,S)$. There exists natural numbers $m$ and $n$, weighted paths $P^i=(f^i,n,\underline{t},\underline{U^i},\underline{g^i})$ ($i=1,\ldots,m$) and $P'^i=(f^i,n,\underline{t},\underline{U'^i},\underline{g'^i})$ ($i=0,\ldots,m-1$) with $f^0=f$, $f^1=f'$, $\underline{U'^{i-1}}=\underline{U^i}$, and homotopies $H^i:P'^{i-1}\to P^i$ $(i=1,\ldots,m)$.
\end{lemma}
\begin{proof}
	$[f]=[f']$ means that there exists $H:[0,1]\times[0,1]\to B$ satisfying condition (a) of Definition 2.9. Then, according to Lemma 6.8.5 in \cite{BoJ}, there exist finite sequences $x_0,\ldots,x_m$ and $y_0,\ldots,y_m$ of real numbers with
	\begin{enumerate}[$\bullet$]
		\item $0=x_0<\ldots<x_m=1$ and $0=y_0<\ldots<y_n=1$,
		\item for every $(i,j)\in\{1,\ldots,m\}\times\{1,\ldots,n\}$, there exist $U_{(i,j)}\in\mathcal{U}$ with
		\begin{equation*}
		H([y_{j-1},y_j]\times[x_{i-1},x_i])\subseteq U_{(i,j)}.
		\end{equation*}
	\end{enumerate}
	Using these sequences we can define
	\begin{enumerate}[$\bullet$]
		\item $f^i$ by $f^i(y)=H(y,x_i)$ ($i=0,\ldots,m$),
		\item $\underline{t}$ by $t_j=y_j$ ($j=0,\ldots,n$),
		\item $\underline{U^i}$ by $\underline{U^i}_j=U_{(i,j)}$ ($i=1,\ldots,m$ and $j=1,\ldots,n$),
		\item $\underline{U'^i}$ by $\underline{U'^i}_j=U_{(i+1,j)}$ ($i=0,\ldots,m-1$ and $j=1,\ldots,n$),
	\end{enumerate}
	and choose $\underline{g^0},\ldots,\underline{g^m}$ and $\underline{g'^0},\ldots,\underline{g'^m}$ arbitrarily (see the last part of Lemma 2.4). It remains to note that $H$ induces homotopies $H^i:P'^{i-1}\to P^i$ ($i=1,\ldots,m$) via $[x_{i-1},x_i]\approx[0,1]$.
\end{proof}

Now, following Lemmas \ref{thm:Delta-Identity1} and \ref{thm:Delta-Identity2}, we make their strongest version:

\begin{lemma}
\label{thm:Delta-Identity4}%
	Let $(f,n,\underline{t},\underline{U},\underline{g})$ and $(f',n',\underline{t'},\underline{U'},\underline{g'})$ be weighted paths with $[f]=[f']$ in $\pi_1(B,S)$, and $(h_n,\ldots,h_1)$ and $(h'_{n'},\ldots,h'_1)$ their associated sequences of morphisms in $\pi_1(E,S)$. Then $\delta(h_n)\ldots\delta(h_1)=\delta(h'_{n'})\ldots\delta(h'_1)$.
\end{lemma}
\begin{proof}
	Let us write $P=(f,n,\underline{t},\underline{U},\underline{g})$, $\delta(h_n)\ldots\delta(h_1)=\delta[P]$, and use similar notation for other weighted paths. Using weighted paths $P^i$ ($i=1,\ldots,m$) and $P'^i$ ($i=0,\ldots,m-1$) from Lemma 2.11, we see that
	\begin{enumerate}[$\bullet$]
		\item $\delta[P]=\delta[P'^0]$ by Lemma \ref{thm:Delta-Identity2};
		\item $\delta[P'^0]=\delta[P^1]$ by Lemma \ref{thm:Delta-Identity3};
		\item $\delta[P^1]=\delta[P'^1]$ by Lemma \ref{thm:Delta-Identity2} again, and so on.
	\end{enumerate}
	This gives $\delta[P]=\delta[P']$, as desired.
\end{proof}

Lemma \ref{thm:Delta-Identity4} easily implies that sending $[f]$ to $\delta(h_n)\ldots\delta(h_1)$, where $(h_n,\ldots,h_1)$ is as in Definition \ref{def:assocseq}, determines a functor $\varepsilon:\pi_1(B,S)\to G$, and, to prove Theorem \ref{thm:main}, it remains to prove that it is the unique functor with $\varepsilon\gamma=\delta$.

To prove the equality $\varepsilon\gamma=\delta$, we chose:
\begin{enumerate}[$\bullet$]
	\item any path $e$ in $E$ with endpoints in $E\times_BS=p^{-1}(S)$;
	\item a triple $(n,\underline{t},\underline{U})$ as in Definition \ref{def:wpath} with $f=p(e)$, $e_i$ denoting the restriction of $e$ on $[t_{i-1},t_i]$, and $f_i$ denoting the map $[t_{i-1},t_i]\to U_i$ induced by $f$;
	\item paths $r_i:[0,1]\to E\times_BE\times_BE$ $(i=1,\ldots,n-1)$ with $r_i(0)=(s_{U_i}f(t_i),e(t_i),s_{U_{i+1}}f(t_i))$ and $r_i(1)$ in the inverse image $(E\times_BE\times_BE)\times_BS$ of $S$ under the canonical map
	\begin{equation*}
	E\times_BE\times_BE\to B,
	\end{equation*}
	for each $i=0,\ldots,n$ (again, such choices are possible by our assumption on path-components of $E\times_BE\times_BE$ in Theorem \ref{thm:main}).
\end{enumerate}
Then, using the weighted path $(f,n,\underline{t},\underline{U},\underline{g})$, in which $g_i=\langle q_1r_i,q_3r_i\rangle$, and its associated sequence $h_n,\ldots,h_1$, we can write:
\begin{equation*}
\begin{array}{rcl}
\varepsilon\gamma([e]) & = & \varepsilon([f])=\delta(h_n)\ldots\delta(h_1) \\
	& = & \delta([s_{U_n}f_n][p_2g_{n-1}]^{-1})\ldots\delta([p_1g_i][s_{U_i}f_i][p_2g_{i-1}]^{-1})\ldots\delta([p_1g_1][s_{U_1}f_1]) \\
	& = & \delta([s_{U_n}f_n][q_3r_{n-1}]^{-1})\ldots\delta([q_1r_i][s_{U_i}f_i][q_3r_{i-1}]^{-1})\ldots\delta([q_1r_1][s_{U_1}f_1]).
\end{array}
\end{equation*}
In order to complete this calculation, note that:
\begin{enumerate}[(a)]
	\item  $ps_{U_i}f_i=pe_i$, $pq_3r_i=pq_2r_i$, and $pq_1r_i=pq_2r_i$;
	\item  as follows from (a), $\langle s_{U_i}f_i,e_i\rangle$, $\langle q_3r_i,q_2r_i\rangle$, and $\langle q_1r_i,q_2r_i\rangle$ are paths in $E$;
	\item  thanks to (b), we can write:
	\begin{equation*}
	\resizebox{.945\textwidth}{!}{$
	\begin{array}{rcl}
	\delta([s_{U_n}f_n][q_3r_{n-1}]^{-1}) & = & \delta\alpha([\langle s_{U_n}f_n,e_n\rangle][\langle q_3r_{n-1},q_2r_{n-1}\rangle]^{-1}) \\
		& = & \delta\beta([\langle s_{U_n}f_n,e_n\rangle][\langle q_3r_{n-1},q_2r_{n-1}\rangle]^{-1})=\delta([e_n][q_2r_{n-1}]^{-1}), \\
	\delta([q_1r_i][s_{U_i}f_i][q_3r_{i-1}]^{-1}) & = & \delta\alpha([\langle q_1r_i,q_2r_i\rangle][\langle s_{U_i}f_i,e_i\rangle][\langle q_3r_{i-1},q_2r_{i-1}\rangle]^{-1}) \\
		& = & \delta\beta([\langle q_1r_i,q_2r_i\rangle][\langle s_{U_i}f_i,e_i\rangle][\langle q_3r_{i-1},q_2r_{i-1}\rangle]^{-1})=\delta([q_2r_i][e_i][q_2r_{i-1}]^{-1}), \\
	\delta([q_1r_1][s_{U_1}f_1]) & = & \delta\alpha([\langle q_1r_1,q_2r_1\rangle][\langle s_{U_1}f_1,e_1\rangle]) \\
		& = & \delta\beta([\langle q_1r_1,q_2r_1\rangle][\langle s_{U_1}f_1,e_1\rangle])=\delta([q_2r_1][e_1]);
	\end{array} $}
	\end{equation*}
	\item  the sequence $[e_n][q_2r_{n-1}]^{-1},\ldots,[q_2r_i][e_i][q_2r_i]^{-1},\ldots,[q_2r_1][e_1]$ is a composable sequence of morphisms in $\pi_1(E,S)$, and so
	\begin{equation*}
	\begin{array}{l}
	\delta([e_n][q_2r_{n-1}]^{-1})\ldots\delta([q_2r_i][e_i][q_2r_i]^{-1})\ldots\delta([q_2r_1][e_1]) \\
	\qquad =\delta([e_n][q_2r_{n-1}]^{-1}\ldots[q_2r_i][e_i][q_2r_i]^{-1}\ldots[q_2r_1][e_1])=\delta([e_n]\ldots[e_1])=\delta([e]).
	\end{array}
	\end{equation*}
\end{enumerate}
That is,
\begin{equation*}
\begin{array}{rcl}
\varepsilon\gamma([e]) & = & \delta([s_{U_n}f_n][q_3r_{n-1}]^{-1})\ldots\delta([q_1r_i][s_{U_i}f_i][q_3r_{i-1}]^{-1})\ldots\delta([q_1r_1][s_{U_1}f_1]) \\
	& = & \delta([e_n][q_2r_{n-1}]^{-1})\ldots\delta([q_2r_i][e_i][q_2r_i]^{-1})\ldots\delta([q_2r_1][e_1])=\delta([e])
\end{array}
\end{equation*}
(where the second equality follows from (c)), which proves that $\varepsilon$ satisfies the equality $\varepsilon\gamma=\delta$.

The uniqueness is easy: for any functor $\varepsilon:\pi_1(B,S)\to G$ with $\varepsilon'\gamma=\delta$, and any path $f$ in $\pi_1(B,S)$, we have:
\begin{equation*}
\varepsilon'([f])=\varepsilon'(\gamma(h_n)\ldots\gamma(h_1))=\varepsilon'\gamma(h_n)\ldots\varepsilon'\gamma(h_1))=\delta(h_n)\ldots\delta(h_1)=\varepsilon([f]),
\end{equation*}
where:
\begin{enumerate}[$\bullet$]
	\item $(h_n,\ldots,h_1)$ is the sequence of morphisms in $\pi_1(E,S)$ associated with any weighted path constructed for $f$, which does exist by Lemma \ref{lem:weight}.
	\item the first equality follows from Lemma \ref{lem:comp}(b).
\end{enumerate}
This completes our proof of Theorem  \ref{thm:main}.

\section{Additional remarks}
\label{sec:AdditionalRemarks}

In what follows, $\mathbf{Set}$, $\mathbf{Top}$, and $\mathbf{Grpd}$ denote the categories of sets, of topological spaces, and of groupoids, respectively; $\mathbf{\Delta}$ denotes the simplicial category, and, accordingly, $\mathbf{Set}^{\mathbf{\Delta}^{\mathrm{op}}}$ denotes the category of simplicial sets.

Taking $S=B$ in Theorem \ref{thm:main} we obtain:
\begin{corollary}
\label{cor:1}
	Let $p:E\to B$ be a locally sectionable continuous map of topological spaces. There is a coequalizer diagram
	\begin{equation*}\xymatrix{\pi_1(E\times_BE)\ar@<0.5ex>[r]^-{\alpha}\ar@<-0.5ex>[r]_-{\beta}&\pi_1(E)\ar[r]^-{\gamma}&\pi_1(B)},
	\end{equation*}
	in $\mathbf{Grpd}$, in which the functors $\alpha$, $\beta$, and $\gamma$ are induced respectively by the first and the second projection $E\times_BE\to E$, and by $p$.
\end{corollary}

\begin{remark}
\label{rem:1}%
	Let us mention a couple of much easier counterparts of Corollary \ref{cor:1}:
	\begin{enumerate}[(a)]
		\item Suppose $p:E\to B$ itself is sectionable, that is, $p$ is a split epimorphism in $\mathbf{Top}$. Then the diagram
		\begin{equation*}\xymatrix{E\times_BE\ar@<0.5ex>[r]\ar@<-0.5ex>[r]&E\ar[r]&B},
		\end{equation*}
		becomes an absolute coequalizer diagram, that is, it is preserved not just by $\pi_1$, but by every functor defined on $\mathbf{Top}$. In particular, this is the case when $p$ is a trivial fibre bundle.
		\item \label{rem:2}%
		More generally, suppose the image $S^\Delta (p)$ of $p:E\to B$ under the singular complex functor $S^\Delta:\mathbf{Top}\to \mathbf{Set}^{\mathbf{\Delta}^{\mathrm{op}}}$ is an epimorphism. Then the diagram of Corollary \ref{cor:1} is again a coequalizer diagram since: (i) the functor $S^\Delta$ being a right adjoint preserves pullbacks, and so
		\begin{equation*}\xymatrix{S^\Delta(E\times_BE)\ar@<0.5ex>[r]\ar@<-0.5ex>[r]&S^\Delta(E)},
		\end{equation*}
		is the kernel pair of $S^\Delta(p)$; (ii) every epimorphism in $\mathbf{Set}^{\mathbf{\Delta}^{\mathrm{op}}}$ is a coequalizer of its kernel pair; (iii) the fundamental groupoid functor $\mathbf{Set}^{\mathbf{\Delta}^{\mathrm{op}}}\to \mathbf{Grpd}$ being a left adjoint preserves coequalizers, and, composed with $S^\Delta$, gives the fundamental groupoid functor $\mathbf{Top}\to\mathbf{Grpd}$. Moreover, our assumption here can be weakened by using certain truncated simplicial sets instead of all simplicial sets.
	\end{enumerate}
\end{remark}

Corollary \ref{cor:1} could be called the `absolute case' of Theorem \ref{thm:main}. The `opposite extreme case', which seems to be a new simple way of calculating the fundamental groups of some topological spaces is:
\begin{corollary}
\label{cor:3}
	Let $p:E\to B$ be a locally sectionable continuous map of topological spaces, and $x$ an element of $B$ such that the inverse image $(E\times_BE\times_BE)\times_B\{x\}$ of $x$ under the canonical map $E\times_BE\times_BE\to B$ meets every path-component of $E\times_BE\times_BE$. Then the diagram
	\begin{equation*}\xymatrix{\pi_1(E\times_BE,\{x\})\ar@<0.5ex>[r]^-{\alpha}\ar@<-0.5ex>[r]_-{\beta}&\pi_1(E,\{x\})\ar[r]^-{\gamma}&\pi_1(B,x),}
	\end{equation*}
	in which the functors $\alpha$, $\beta$, and $\gamma$ are induced by the first and the second projection $E\times_BE\to E$, and by $p$, respectively, is a coequalizer diagram in $\mathbf{Grpd}$.
\end{corollary}

\begin{remark}
	When $B$ is a `good' space admitting a universal covering map $p:E\to B$, the groupoids $\pi_1(E\times_BE)$ and $\pi_1(E)$ are coproducts of indiscrete groupoids. Therefore Theorem \ref{thm:main} and Corollaries \ref{cor:1} and \ref{cor:3} present $\pi_1(B,S)$, $\pi_1(B)$, and $\pi_1(B,x)$, respectively, as colimits of indiscrete groupoids.
\end{remark}

\begin{remark}
	Since the fundamental group(oid) of a space $B$ does not in general classify covering spaces over $B$,   one cannot deduce Theorem \ref{thm:basic vKT} from the results of \cite{BJ} based on categorical Galois theory \cite{J}. Nevertheless,  a detailed comparison of the two approaches would be desirable, especially since all locally sectionable continuous maps are effective descent morphisms in $\mathbf{Top}$. In particular, it would be interesting to answer the following questions:
	\begin{enumerate}[(a)]
		\item  What is a necessary and sufficient condition on $p:E\to B$ under which the diagram considered in Corollary \ref{cor:1} is a coequalizer diagram in $\mathbf{Grpd}$?
		\item  Compare the diagrams considered in Theorem \ref{thm:main}  and in Corollary \ref{cor:1}. In general there is no way to conclude that the first of them is a coequalizer diagram whenever the second one is. What are reasonable conditions on $p:E\to B$ under which such conclusion can be made?
	\end{enumerate}
\end{remark}

\section{Theorem 1.1 almost implies Theorem 2.2}
\label{sec:(1.1)->(2.2)}

Let $\mathbf{Top_2}$ be the category of pairs $(X,A)$, where $X$ is a topological space and $A$ is a subset of $X$. In the notation of Section \ref{sec:MainTheorem}, consider the diagram
\begin{equation*}
\xymatrix{(E\times_BE\times_BE'\times_BE',S)\ar@<0.5ex>[r]\ar@<-0.5ex>[r]\ar@<0.5ex>[d]\ar@<-0.5ex>[d] &
	(E\times_BE'\times_BE',S)\ar@<0.5ex>[d]\ar@<-0.5ex>[d]\ar[r] &
	(E'\times_BE',S)\ar@<0.5ex>[d]\ar@<-0.5ex>[d] \\
(E\times_BE\times_BE',S)\ar[d]\ar@<0.5ex>[r]\ar@<-0.5ex>[r] &
	(E\times_BE',S)\ar[d]\ar[r] &
	(E',S)\ar[d]^{p'}\\(E\times_BE,S)\ar@<0.5ex>[r]\ar@<-0.5ex>[r] &
	(E,S)\ar[r]_p &
	(B,S)}
\end{equation*}
in $\mathbf{Top_2}$, in which:
\begin{enumerate}[$\bullet$]
	\item $E'$ is the coproduct of all elements of $\mathcal{U}$, and $p'$ is induced by all the coproduct injections;
	\item presenting pairs we use the same convention as for the fundamental groupoids, that is, we write $(E,S)$ instead of $(E,E\times_BS)$, etc.
	\item the bottom right-hand square is a pullback and all unlabeled arrows are obviously defined canonical morphisms.	
\end{enumerate}
We observe:
\begin{enumerate}[(a)]
	\item Once the columns of this diagram satisfy the \textit{triple intersection assumption} of Theorem \ref{thm:basic vKT}, each of them is carried to a coequalizer diagram by the fundamental groupoid functor $\mathbf{Top_2}\to\mathbf{Grpd}$.
	\item The first two rows of this diagram are absolute coequalizer diagrams (for the same reason as in Remark \ref{rem:1}(a)).
	\item As follows from (b), the first two rows of this diagram are also carried to a coequalizer diagram by the fundamental groupoid functor $\mathbf{Top_2}\to\mathbf{Grpd}$.
	\item Under the assumption made in (a), Theorem \ref{thm:main} follows from (a) and (c).
\end{enumerate}
That is, under a mild additional condition, Theorem \ref{thm:main} can be deduced from \ref{thm:basic vKT} using the pullback of $p$ and $p'$ above and simple purely-categorical arguments. In particular, this mild additional condition obviously holds when $S=B$.

\newpage
\refs

\bibitem [Borceux-Janelidze, 2001]{BoJ} F. Borceux and G. Janelidze, \emph{Galois Theories}. Cambridge Studies in Advanced Mathematics 72, Cambridge University Press, 2001.
	
\bibitem [Brown, 2006]{B} R. Brown, \emph{Elements of modern topology}. McGraw Hill (1968), Revised version available as \emph{Topology and Groupoids} (2006) from amazon.

\bibitem [Brown, 1987]{BLMS}R. Brown, \emph{From groups to groupoids: a brief survey}.   Bull. London Math. Soc. 19 (1987) 113-134.

\bibitem [Brown, 2018]{Indag} R. Brown, \emph{Modelling and Computing Homotopy Types: I}. Indagationes Math. (Special issue in honor of L.E.J. Brouwer) 29 (2018) 459-482.

\bibitem [Brown-Janelidze, 1997]{BJ} R. Brown and G. Janelidze, \emph{Van Kampen theorem for categories of covering morphisms in lextensive categories}. J. Pure Appl. Algebra 119, 1997, 255-263.

\bibitem [Brown-Salleh, 1984]{BS} R. Brown and A. R. Salleh, \emph{A van Kampen theorem for unions of non-connected spaces}. Arch. Math. 42, 1984, 85-88.

\bibitem [Brown-Higgins-Sivera, 2011]{BHS} R. Brown,  P.~J. Higgins, and R. Sivera,  \emph{Nonabelian algebraic topology: filtered spaces, crossed complexes, cubical homotopy groupoids}. EMS Tracts in Mathematics Vol 15. European Mathematical Society (2011).

\bibitem [Douady-Douady, 1979]{D} A. {D}ouady and R. {D}ouady. \emph{Algebres et th\'eories {G}aloisiennes}. Volume~2. CEDIC, Paris (1979).

\bibitem [Gill-Gillespie-Semeraro, 2018]{GGC} N. Gill,  N.I. Gillespie, \&  J. Semeraro,  \emph{Conway groupoids and completely transitive codes}. Combinatorica 38, 399-442 (2018). \url{https://doi.org/10.1007/s00493-016-3433-7}

\bibitem [Grothendieck, 1984]{GrEsq} A. Grothendieck, \emph{Esquisse d'un Programme} in \emph{Geometric Galois Acitions, 1. Around Grothendieck's Esquisse d'un Programme}. Edited by Leila Schneps (CNRS) and Pierre Lochak (CNRS), Cambridge University Press, London Mathematical Society LNS 242, 5-48 (1997).

\bibitem [Higgins, 2005]{H} P.~J.Higgins, \emph{Notes on Categories and Groupoids}. {Mathematical Studies}, Volume~32. Van Nostrand Reinhold Co. (1971); Reprints in Theory and Applications of Categories, No. 7 (2005) 1-195.

\bibitem [Janelidze, 1991]{J} G. Janelidze, \emph{Precategories and Galois theory}. Proc. Int. Conf. ``Category Theory 1990'', Como (Italy), Lecture Note in Mathematics 1488, Springer 1991, 157-173.

\bibitem [Janelidze-Tholen, 1991]{JT} G. Janelidze and W. Tholen, \emph{How algebraic is the change-of-base functor?} Proc. Int. Conf. ``Category Theory 1990'', Como (Italy), Lecture Note in Mathematics 1488, Springer 1991, 174-186.

\bibitem [Nikolaus-Schweigert, 2011]{NS} T. Nikolaus and C.  Schweigert, \emph{Equivariance in higher geometry}. Adv. Math. 226 (2011), no. 4, 3367-3408.

\endrefs

\end{document}